\documentclass[10pt,english,reqno]{amsart}

\usepackage{amsthm,amsmath,amssymb,amsfonts}
\usepackage{booktabs,float,caption,setspace}
\usepackage{mathtools}
\usepackage[usenames]{color}
\usepackage{a4wide}

\definecolor{webgreen}{rgb}{0,.5,0}
\definecolor{webbrown}{rgb}{.6,0,0}

\usepackage[
  pdfauthor={},
  pdfkeywords={},
  pdftitle={},
  pdfcreator={},
  pdfproducer={},
  linktocpage,colorlinks,bookmarksnumbered,linkcolor=blue,
  citecolor=webgreen,urlcolor=webbrown]{hyperref}

\theoremstyle{plain}
\newtheorem{theorem}{Theorem}
\newtheorem{corollary}[theorem]{Corollary}
\newtheorem{lemma}[theorem]{Lemma}

\theoremstyle{definition}
\newtheorem{example}[theorem]{Example}
\newtheorem{remark}[theorem]{Remark}

\newcommand{\NN}{\mathbb{N}}
\newcommand{\ZZ}{\mathbb{Z}}

\newcommand{\RR}{\mathbb{R}}
\newcommand{\CC}{\mathbb{C}}
\newcommand{\bb}{\mathtt{b}}

\newcommand{\MA}{\mathcal{A}}
\newcommand{\MC}{\mathcal{C}}
\newcommand{\MF}{\mathcal{F}}
\newcommand{\FC}{\mathfrak{c}}
\newcommand{\Cat}{\mathfrak{C}}
\newcommand{\HC}{\mathtt{C}}

\newcommand{\andq}{\quad \text{and} \quad}
\newcommand{\textq}[1]{\quad \text{#1} \quad}
\newcommand{\ndots}{...}
\newcommand{\sims}{\,\,\sim\,\,}
\newcommand{\divsum}{{}\!'\,}

\newcommand{\seqnum}[1]{\href{https://oeis.org/#1}{\rm \underline{#1}}}

\begin{document}

\title[Asymptotic products of binomial and multinomial coefficients revisited]
{Asymptotic products of binomial and multinomial\\ coefficients revisited}
\author{Bernd C. Kellner}
\address{G\"ottingen, Germany}
\email{bk@bernoulli.org}
\subjclass[2020]{11B65 (Primary), 11Y60, 41A60 (Secondary)}
\keywords{Asymptotic constant and formula, Glaisher--Kinkelin constant, factorial,
binomial and multinomial coefficient, central binomial coefficient, Catalan number.}

\begin{abstract}
In this note, we consider asymptotic products of binomial and multinomial
coefficients and determine their asymptotic constants and formulas. Among them,
special cases are the central binomial coefficients, the related Catalan
numbers, and binomial coefficients in a row of Pascal's triangle. For the
latter case, we show that it can also be derived from a limiting case of
products of binomial coefficients over the rows. The asymptotic constants are
expressed by known constants, for example, the Glaisher--Kinkelin constant.
In addition, the constants lie in certain intervals that we determine precisely.
Subsequently, we revisit a related result of Hirschhorn and clarify the given
numerical constant by showing the exact expression.
\end{abstract}

\maketitle


\section{Introduction}

The Glaisher--Kinkelin constant $\MA$ occurs in many asymptotic formulas.
It was independently found by Glaisher~\cite{Glaisher:1878}
and Kinkelin~\cite{Kinkelin:1860} by considering the asymptotic formula of the
\textit{hyperfactorial}:
\begin{equation} \label{eq:GK}
  \prod_{\nu=1}^n \nu^\nu \sims \MA \times
  n^{\binom{n+1}{2}+\frac{1}{12}} \, e^{-\frac{1}{4} n^2}
  \quad \text{as\ } n \to \infty.
\end{equation}
The notation $c \times f(x)$ means here that $c$ is an asymptotic constant,
$f$ is an increasing function, and $\log f(x)$ has no constant term for $x \to \infty$.
Regarding products of factorials, the author~\cite[Theorem 12]{Kellner:2009}
showed for integers $k \geq 1$ the asymptotic formula
\begin{equation} \label{eq:facFk}
  \prod_{\nu=1}^n (k\nu)!  \sims
  \MF_k \, \MA^k \, (2\pi)^{\frac{1}{4}} \times
  \left( k e^{-\frac{3}{2}} n \right)^{k \binom{n+1}{2}}
  \left( 2\pi k e^{\frac{k}{2}-1} n \right)^{\frac{n}{2}}
  n^{\frac{k^2+3k+1}{12k}} \quad \text{as\ } n \to \infty
\end{equation}
with certain constants $\MF_k$ obeying the limit behavior that
$\lim\limits_{k \to \infty} \MF_k = 1$.

The main purpose of this paper is to discuss certain types of asymptotic products
of binomial and multinomial coefficients and to determine their asymptotic constants
and formulas. It turns out that formulas~\eqref{eq:GK} and~\eqref{eq:facFk}
play a major role in this context, as well as the constants $\MA$ and $\MF_k$.

The next section shows further properties of the constants~$\MF_k$.
Section~\ref{sec:main} contains the main results and supplementary examples.
The subsequent section gives their proofs. In the last section, we revisit a
related result of Hirschhorn~\cite{Hirschhorn:2013}.
Well-known results, which we use implicitly, can be found, e.g., in the book of
Graham et~al.~\cite{GKP:1994}. The numerical values were computed by
\textsl{Mathematica} with 30 decimal digit accuracy and truncated to 25 digits
after the decimal point. All given intervals are the best possible throughout
the paper.


\section{Properties of the constants \texorpdfstring{$\MF_k$}{}}

For a divergent series expansion of a function $f: \RR^+ \to \RR$,
we use the notation
\[
  f(x) = \sum_{\nu \geq 1} \divsum f_\nu(x)
  = \sum_{\nu=1}^{m-1} f_\nu(x) + \theta_m(x) f_m(x),
\]
where the sum is truncated at a suitable index ${m \geq 1}$ and ${\theta_m(x) \in (0,1)}$.
Let $B_n$ be the $n$th Bernoulli number defined by
\[
  \frac{z}{e^z-1} = \sum_{n=0}^\infty B_n \frac{z^n}{n!} \quad (|z| < 2 \pi).
\]
The Glaisher--Kinkelin constant $\MA \approx 1.2824271291006226368753425\ndots$
is given by
\[
  \log \MA = \tfrac{1}{12} - \zeta'(-1)
  = \tfrac{1}{12}(\gamma + \log (2\pi)) - \tfrac{\zeta'(2)}{2\pi^2},
\]
where $\gamma = -\Gamma'(1) \approx 0.5772156649015328606065120\ndots$
is Euler’s constant, $\zeta$ is the Riemann zeta function, and $\Gamma$ is the
gamma function. The constants $\MF_k$ can be expressed as follows.

\begin{theorem}[Kellner~{\cite[Theorems 12 and 13]{Kellner:2009}}] \label{thm:kel}
For $k \geq 1$, we have
\begin{align}
  \log \MF_k &= \tfrac{k}{4} \log (2\pi) - \tfrac{k^2+1}{k} \log \MA
  + \tfrac{1}{12k}(1 - \log k)
  - \sum_{\nu=1}^{k-1} \tfrac{\nu}{k} \log \Gamma(\tfrac{\nu}{k}) \label{eq:logFk1} \\
    &= \frac{\gamma}{12k} + \sum_{j \geq 2} \divsum
  \frac{B_{2j} \, \zeta(2j-1)}{2j (2j-1) \, k^{2j-1}}. \label{eq:logFk2}
\end{align}
\end{theorem}

Clearly, it follows from \eqref{eq:logFk2} that $\lim\limits_{k \to \infty} \MF_k = 1$
as mentioned before (for further properties of $\MF_k$, see \cite[Theorem 12]{Kellner:2009}).
We can make this limit behavior even more precise (cf. Table~\ref{tbl:Fk}).

\begin{corollary} \label{cor:seq}
The sequence $(\MF_k)_{k \geq 1}$ is strictly decreasing with limit $1$.
\end{corollary}

\begin{proof}
Let $k \geq 1$. By Theorem~\ref{thm:kel}, we have a truncated expansion
\begin{equation} \label{eq:logFk3}
  \log \MF_k = \frac{\gamma}{12k} + \widetilde{\theta}_k \frac{B_4 \, \zeta(3)}{12k^3}
  \textq{with} \widetilde{\theta}_k \in (0,1).
\end{equation}
Note that ${B_4 = -\frac{1}{30}}$ and the remainder term is relatively small.
We compare the infimum and supremum of the expansion of $\log \MF_k$ and $\log \MF_{k+1}$,
respectively. Since
\[
  \frac{\gamma}{12k} - \frac{\zeta(3)}{360k^3} > \frac{\gamma}{12(k+1)}
  \textq{is equivalent to}
  \frac{k^2}{k+1} > \frac{\zeta(3)}{30\gamma} \approx 0.069416\ndots,
\]
it follows that $\log \MF_k > \log \MF_{k+1}$ for all $k \geq 1$,
implying the result.
\end{proof}

We need the following lemmas later on.

\begin{lemma} \label{lem:estim}
Define $g(a,b) = \frac{1}{a} + \frac{1}{b} - \frac{1}{a+b}$.
For $a, b \in \NN$ and $a+b \geq 3$, we have $\frac{7}{6} \geq g(a,b) > 0$.
\end{lemma}

\begin{proof}
Note that $g(a,b) = \frac{1}{a} + \frac{a}{b(a+b)} > 0$ for $a,b \geq 1$.
Further, we have $g(a,b) = \frac{1}{a} + \frac{1}{b} - \frac{1}{a+b} < 1$ for
$a,b \geq 2$. By symmetry, there remains the case $a=1$ and $a+b \geq 3$.
Thus, $g(a,b) = 1 + \frac{1}{b(b+1)}$ has its maximum at $b=2$, showing that
$g(a,b) \leq 1 + \frac{1}{6}$.
\end{proof}

\begin{lemma} \label{lem:interval}
Let $a > b \geq 1$ and $\FC_{a,b} = \frac{\MF_a}{\MF_{b} \, \MF_{a-b}}$.
Then $\FC_{a,b} \in [\FC_{2,1}, 1)$.
\end{lemma}

\begin{proof}
Let $a > b \geq 1$. By Corollary~\ref{cor:seq},
we have $\min(\MF_{b},\MF_{a-b}) > \MF_{a} > 1$ and so $\FC_{a,b} < 1$.
Letting $a = 2b$, we get the supremum by
$\lim\limits_{b \to \infty} \FC_{a,b} = \lim\limits_{b \to \infty} \MF_{2b}/\MF_b^2 = 1$.
Now, set $\alpha = \frac{\gamma}{12}$ and $\beta = \frac{\zeta(3)}{360}$.
Following the proof of Corollary~\ref{cor:seq}, we define
$\log \MF^+_k = \frac{\alpha}{k}$ and
$\log \MF^-_k = \frac{\alpha}{k} - \frac{\beta}{k^3}$
as the supremum and infimum of \eqref{eq:logFk3}, respectively.
On the one side, we have
\[
  \log \FC_{2,1} < \log \MF^+_2 - 2 \log \MF^-_1
  = - \tfrac{3}{2} \alpha + 2 \beta = A.
\]
Let $a \geq 3$ and $a > b \geq 1$, so $\FC_{a,b} \neq \FC_{2,1}$.
On the other side, we have
\[
  \log \FC_{a,b} > \log \MF^-_a - \log \MF^+_b - \log \MF^+_{a-b}
  = ( \tfrac{1}{a} - \tfrac{1}{b} - \tfrac{1}{a-b} ) \, \alpha  - a^{-3} \beta = B.
\]
We show that $B > A$, which implies that $\FC_{a,b} > \FC_{2,1}$.
The inequality turns into
\begin{equation} \label{eq:estim}
  B' = \tfrac{3}{2} + \tfrac{1}{a} - \tfrac{1}{b} - \tfrac{1}{a-b}
  > ( 2 + a^{-3} ) \tfrac{\beta}{\alpha} = A'.
\end{equation}
Using Lemma~\ref{lem:estim} with parameters $a-b$ and $b$, we infer that
$B' = \tfrac{3}{2} - g(a-b,b) \geq \tfrac{3}{2} - \tfrac{7}{6} = \tfrac{1}{3}$.
Since $3 \tfrac{\beta}{\alpha} \approx 0.208250\ndots > A'$,
this implies \eqref{eq:estim} and shows the result.
\end{proof}

The first few constants $\MF_k$ can be evaluated as in Table~\ref{tbl:Fk} below.
By Euler's reflection formula
\begin{align*}
  \Gamma(z) \, \Gamma(1-z) &= \frac{\pi}{\sin(\pi z)} \quad (z \in \CC \setminus \ZZ) \\
\intertext{and Legendre's duplication formula}
  \Gamma(z) \, \Gamma(z+\tfrac{1}{2}) &= 2^{1-2z} \sqrt{\pi} \, \Gamma(2z),
\end{align*}
one can reduce terms of the form $\Gamma(\tfrac{\nu}{k})$ that occur in $\MF_k$ using \eqref{eq:logFk1}.
Let $\phi = \frac{\sqrt{5}+1}{2}$ be the golden ratio.

\begin{table}[H] \small
\setstretch{1.5}
\begin{center}
\begin{tabular}{r@{\;=\;}ll}
  \toprule
  \multicolumn{2}{c}{Constant} & \multicolumn{1}{c}{Value} \\\hline
  $\MF_1$ & $(2\pi)^{\frac14} \, e^{\frac{1}{12}} \MA^{-2}$
  & $1.0463350667705031809809506\ndots$ \\
  $\MF_2$ & $(2\pi)^{\frac14} \, 2^{\frac{5}{24}} \, e^{\frac{1}{24}} \MA^{-\frac{5}{2}}$
  & $1.0239374116371184015779507\ndots$ \\
  $\MF_3$ & $(2\pi)^{\frac{1}{12}} \, 3^{\frac{11}{36}} \, e^{\frac{1}{36}} \,
    \Gamma(\frac{1}{3})^{\frac{1}{3}} \MA^{-\frac{10}{3}} $
  & $1.0160405370646209912870365\ndots$ \\
  $\MF_4$ & $2^{\frac{7}{12}} \, e^{\frac{1}{48}} \Gamma(\frac{1}{4})^{\frac{1}{2}}
    \MA^{-\frac{17}{4}}$
  & $1.0120458980239446462423302\ndots$ \\
  $\MF_5$ & $(2\pi)^{-\frac{3}{20}} \, 5^{\frac{1}{3}} \, \phi^{-\frac{1}{10}} \, e^{\frac{1}{60}}
    \Gamma(\frac{1}{5})^{\frac{3}{5}} \Gamma(\frac{2}{5})^{\frac{1}{5}}
    \MA^{-\frac{26}{5}}$
  & $1.0096399728364770508687282\ndots$ \\
  $\MF_6$ & $(2\pi)^{-\frac{7}{12}} \, 2^{\frac{25}{72}} \, 3^{\frac{47}{72}} \, e^{\frac{1}{72}}
    \Gamma(\frac{1}{3})^{\frac{5}{3}} \MA^{-\frac{37}{6}}$
  & $1.0080336272420732654455927\ndots$ \\
  \bottomrule
\end{tabular}
\vspace*{-6pt}
\caption{First few values of $\MF_k$.}
\label{tbl:Fk}
\end{center}
\end{table}


\section{Main results and examples}
\label{sec:main}

We first consider asymptotic products of multinomial coefficients.
It becomes apparent that the occurring asymptotic constants are intimately
related to the constants~$\MF_k$.

\begin{theorem} \label{thm:main}
Let $r \geq 2$. Let $a = b_1 + \cdots + b_r$ and $\bb = (b_1, \ldots, b_r)$
with $b_\nu \geq 1$. Asymptotically, we have
\[
  \prod_{n=1}^m \binom{a n}{\bb n} =
  \prod_{n=1}^m \binom{a n}{b_1 n, \, b_2 n, \, \cdots\!, \, b_r n}
  \sims \MC_{a,\bb} \times \frac{P_a(m)}{P_{b_1}(m) \cdots P_{b_r}(m)}
  \quad \text{as\ } m \to \infty
\]
with the functions
\[
  P_k(x) = \left( k^k \right)^{\binom{x+1}{2}}
  \left( \tfrac{2\pi k}{e} \, x \right)^{\frac{x}{2}}
  x^{\frac{3k+1}{12k}} \quad (k \geq 1)
\]
and the asymptotic constant
\[
  \MC_{a,\bb} = \FC_{a,\bb} \, (2\pi)^{\frac{1}{4}(1-r)}, \textq{where}
  \FC_{a,\bb} = \frac{\MF_a}{\MF_{b_1} \cdots \MF_{b_r}} \in (0,1).
\]
\end{theorem}

The asymptotic products of binomial coefficients can be derived as a special case.
The occurring asymptotic constants lie in relatively small intervals.

\begin{theorem} \label{thm:main2}
Let $a > b \geq 1$ and $c=a-b$. Asymptotically, we have
\[
  \prod_{n=1}^m \binom{an}{bn} \sims \MC_{a,b} \times P_{a,b}(m)
  \quad \text{as\ } m \to \infty
\]
with the function
\[
  P_{a,b}(x) = \left( \tfrac{a^a}{b^b \, c^c} \right)^{\binom{x+1}{2}}
  x^{\frac{1}{12}(\frac{1}{a}-\frac{a}{bc})-\frac{1}{4}} \Big/
  \left(\tfrac{2\pi}{e}\tfrac{bc}{a} x\right)^{\frac{x}{2}}
\]
and the asymptotic constant
\[
  \MC_{a,b} = \FC_{a,b} \, (2\pi)^{-\frac{1}{4}} \in [\MC_{2,1}, (2\pi)^{-\frac{1}{4}}),
  \textq{where}
  \FC_{a,b} = \frac{\MF_a}{\MF_{b} \, \MF_{c}} \in [\FC_{2,1}, 1)
\]
and
\begin{alignat*}{2}
  \FC_{2,1} =\;& \MF_2 / \MF^2_1 &&\approx 0.9352589011148368571152882\ndots, \\
  & (2\pi)^{-\frac{1}{4}} &&\approx 0.6316187777460647012900105\ndots, \\
  \MC_{2,1} =\;& \FC_{2,1} (2\pi)^{-\frac{1}{4}} &&\approx 0.5907270839982808449347463\ndots.
\end{alignat*}
\end{theorem}

As an application, we find formulas for the related products of the central binomial
coefficients $\binom{2n}{n}$, as well as of the Catalan numbers
$\Cat_n = \frac{1}{n+1} \binom{2n}{n}$.

\begin{corollary} \label{cor:cat}
Asymptotically, we have
\[
  \prod_{n=1}^m \binom{2n}{n} \sims \MC_{2,1} \times
  \frac{2^{m^2} \bigl( \frac{4e}{\pi} \bigr)^{\!\!\frac{m}{2}}}{m^{\frac{m}{2} + \frac{3}{8}}} \andq
  \prod_{n=1}^m \Cat_n \sims \MC_{\mathrm{Cat}} \times
  \frac{2^{m^2} \bigl( \frac{4 e^3}{\pi} \bigr)^{\!\!\frac{m}{2}}}{m^{\frac{3}{2}m + \frac{15}{8}}}
  \quad \text{as\ } m \to \infty
\]
with
\begin{alignat*}{2}
  \MC_{2,1} &= \frac{2^{\frac{5}{24}}\,\MA^{\frac{3}{2}}}{(2\pi)^{\frac{1}{2}}\,e^{\frac{1}{8}}}
    &&\approx 0.5907270839982808449347463\ndots \\
\shortintertext{and}
  \MC_{\mathrm{Cat}} &= \MC_{2,1} (2\pi)^{-\frac{1}{2}}
    &&\approx 0.2356660099851628316196795\ndots,
\end{alignat*}
respectively.
\end{corollary}

\begin{remark}
See sequences \seqnum{A007685} and \seqnum{A003046} in the OEIS~\cite{OEIS} for products
of consecutive central binomial coefficients and Catalan numbers, respectively.
Similar finite products appear in a table of Gould~\cite{Gould:1994}.
Zeilberger~\cite{Zeilberger:1999} gave a short proof of a conjecture of Chan, Robbins,
and Yuen that these products of the Catalan numbers are connected to volumes of certain
polytopes.
\end{remark}

The next example shows a further computation, where some gamma factors occur.

\begin{example}
Asymptotically, we have
\[
  \prod_{n=1}^m \binom{5n}{2n} \sims \MC_{5,2} \times
  \frac{\left(\frac{5^5}{2^2 \, 3^3}\right)^{\binom{m+1}{2}}}
  {\left(\frac{2\pi}{e}\frac{6}{5}m\right)^{\frac{m}{2}} m^{\frac{109}{360}}}
  \quad \text{as\ } m \to \infty,
\]
where
\[
  \MC_{5,2} = \frac{5^{\frac{1}{3}}}{2^{\frac{5}{24}} 3^{\frac{11}{36}}} \,
  \frac{\MA^{\frac{19}{30}}}{(2\pi)^{\frac{11}{15}} e^{\frac{19}{360}} \phi^{\frac{1}{10}}} \,
  \frac{\Gamma(\frac{1}{5})^{\frac{3}{5}} \Gamma(\frac{2}{5})^{\frac{1}{5}}}
  {\Gamma(\frac{1}{3})^{\frac{1}{3}}} \approx 0.6129670404054601065382712\ndots.
\]
Coincidentally, the exponent $\tfrac{109}{360}$ has a simple continued fraction expansion:
\[
  \tfrac{109}{360} = [0;3,3,3,3,3]
  = \tfrac{1}{3+\tfrac{1}{3+\tfrac{1}{3+\tfrac{1}{3+\tfrac{1}{3}}}}}.
\]
\end{example}

The asymptotic product of binomial coefficients in a row of Pascal's triangle is given,
as follows.

\begin{theorem} \label{thm:bin}
Asymptotically, we have
\[
  \prod_{\nu=0}^n \binom{n}{\nu} \sims \MC_{\mathrm{row}} \times P_{\mathrm{row}}(n)
  \quad \text{as\ } n \to \infty,
\]
where
\[
  P_{\mathrm{row}}(x) = \frac{e^{\frac{x^2}{2}+x}}{(2\pi)^{\frac{x}{2}} \,
  x^{\frac{x}{2}+\frac{1}{3}}}
\]
and
\[
  \MC_{\mathrm{row}} = \left( \MF_1 \, (2\pi)^{\frac{1}{4}} \right)^{-1}
  = \frac{\MA^2}{(2\pi)^{\frac{1}{2}} \, e^{\frac{1}{12}}}
  \approx 0.6036486760360103196707021\ndots.
\]
\end{theorem}

By Theorem~\ref{thm:main2}, the constants $\MC_{a,b}$ lie in the interval
$[0.590727\ndots, 0.631618\ndots)$.
It is no coincidence that $\MC_{\mathrm{row}}$ also lies in this interval.
As shown below, Theorem~\ref{thm:bin} is a limiting case of Theorem~\ref{thm:main2}.
To make the products comparable, we need an additional factor for convergence.

\begin{theorem} \label{thm:bin2}
Let $n \geq 1$. We have
\[
  \lim_{a \to \infty} \MC_{a,1} = \MC_{\mathrm{row}} \andq
  \lim_{a \to \infty} a^{-\binom{n+1}{2}} \prod_{\nu=1}^n \binom{a \nu}{\nu}
  = \prod_{\nu=1}^n \binom{n}{\nu},
\]
implying that
\[
  \lim_{a \to \infty} a^{-\binom{x+1}{2}} P_{a,1}(x) = P_{\mathrm{row}}(x).
\]
\end{theorem}


\section{Proofs}

\begin{proof}[Proof of Theorem~\ref{thm:main}]
We rewrite the terms of \eqref{eq:facFk} for $m \to \infty$, as follows:
\[
  \prod_{n=1}^m (kn)! \sims c_k \times f_k(m).
\]
Let $r \geq 2$, and let $a = b_1 + \cdots + b_r$ and $\bb = (b_1, \ldots, b_r)$
with $b_\nu \geq 1$. By definition, we get
\[
  \prod_{n=1}^m \binom{a n}{b_1 n, \, b_2 n, \, \cdots\!, \, b_r n} \sims
  \frac{c_a}{c_{b_1} \cdots c_{b_r}} \times \frac{f_a(m)}{f_{b_1}(m) \cdots f_{b_r}(m)}.
\]
Since $a - b_1 - \cdots - b_r = 0$, terms of the form $\omega^\tau$ vanish in the
fractions above, when $\omega$ is independent of~$\tau$, and $\tau$ is a placeholder
for $a$ and $b_1,\ldots,b_r$. Thus, we infer from \eqref{eq:facFk} that
\[
  \MC_{a,\bb} = \frac{c_a}{c_{b_1} \cdots c_{b_r}}
  = \FC_{a,\bb} \, (2\pi)^{\frac{1}{4}(1-r)} \textq{with}
  \FC_{a,\bb} = \frac{\MF_a}{\MF_{b_1} \cdots \MF_{b_r}}.
\]
Further, this provides
\[
  \frac{f_a(m)}{f_{b_1}(m) \cdots f_{b_r}(m)}
  = \frac{P_a(m)}{P_{b_1}(m) \cdots P_{b_r}(m)},
\]
where terms are canceled out such that
\[
  P_k(x) = \left( k^k \right)^{\binom{x+1}{2}}
  \left( \tfrac{2\pi k}{e} \, x \right)^{\frac{x}{2}}
  x^{\frac{3k+1}{12k}} \quad (k \geq 1).
\]
It remains to show the interval $(0,1)$ for $\FC_{a,\bb}$.
By Corollary~\ref{cor:seq}, we have $1 > \FC_{a,\bb} > 0$, since $a > b_\nu$ and
$\MF_{b_\nu} > \MF_a > 1$ for $1 \leq \nu \leq r$. We consider two limiting cases
for the infimum and supremum. First, for $r \geq 2$,
let $a = r$ and $\bb = (1, \ldots, 1)$. Then we have
$\lim\limits_{r \to \infty} \FC_{a,\bb} = \lim\limits_{r \to \infty} \MF_r/\MF_1^r = 0$.
Second, Lemma~\ref{lem:interval} shows for $r=2$ the supremum~$1$.
This completes the proof of the theorem.
\end{proof}

\begin{proof}[Proof of Theorem~\ref{thm:main2}]
Let $a > b \geq 1$ and $c=a-b$. Since $\binom{an}{bn} = \binom{an}{bn, cn}$,
we use Theorem~\ref{thm:main} with parameters $r=2$, $a \geq 2$, and $\bb=(b,c)$
to get the result. The computation of $P_{a,b}(x) = P_a(x) / ( P_b(x) P_c(x))$
follows directly. The interval $[\FC_{2,1}, 1)$ for $\FC_{a,b}$ is given by
Lemma~\ref{lem:interval}, and the value of $\FC_{2,1}$ is derived from
Table~\ref{tbl:Fk}.
\end{proof}

\begin{proof}[Proof of Corollary~\ref{cor:cat}]
We first consider the product of central binomials. Therefore, we have $a=2$
and $b=1$. By Table~\ref{tbl:Fk} and the function $P_{2,1}$, we compute by
Theorem~\ref{thm:main2} that
\begin{equation} \label{eq:c21}
  \MC_{2,1} = \frac{\MF_2}{\MF_{1}^2} \, (2\pi)^{-\frac{1}{4}}
  = \frac{2^{\frac{5}{24}}\,\MA^{\frac{3}{2}}}{(2\pi)^{\frac{1}{2}}\,e^{\frac{1}{8}}}
  \andq P_{2,1}(m) = \frac{2^{m^2}
  \bigl( \frac{4e}{\pi} \bigr)^{\!\!\frac{m}{2}}}{m^{\frac{m}{2} + \frac{3}{8}}}.
\end{equation}
For the product of the Catalan numbers $\Cat_n$, we have to modify the terms
in~\eqref{eq:c21} by considering the extra factor $1/(m+1)!$.
By Stirling's approximation, we have
$m! \sim \sqrt{2\pi m} \, (\tfrac{m}{e})^m$ for $m \to \infty$. Thus,
\[
  (m+1)! \sims (m+1) \sqrt{2\pi m} \, (\tfrac{m}{e})^m
  \sims \sqrt{2\pi} \, m^{\frac{3}{2}} \, (\tfrac{m}{e})^m,
\]
noting that $m+1 = m \, (1 + \frac{1}{m})$. Finally, we obtain the related terms
\[
  \MC_{\mathrm{Cat}} = \MC_{2,1} (2\pi)^{-\frac{1}{2}} \andq
  P_{2,1}(m) \, \frac{e^m}{m^{m + \frac{3}{2}}}
  = \frac{2^{m^2} \bigl( \frac{4 e^3}{\pi} \bigr)^{\!\!\frac{m}{2}}}
  {m^{\frac{3}{2}m + \frac{15}{8}}}. \qedhere
\]
\end{proof}

Let $(n)_{\nu}$ denote the falling factorial such that $\binom{n}{\nu} = (n)_{\nu} / \nu!$.
We can neglect the term $\binom{n}{0} = 1$ below.
We give a very short and elegant proof of Theorem~\ref{thm:bin} by using results
of the former sections, as follows.

\begin{proof}[Proof of Theorem~\ref{thm:bin}]
Let $n \geq 1$. First, note that
\begin{equation} \label{eq:facpow}
  \prod_{\nu=1}^n (n)_{\nu} = \prod_{\nu=1}^n \nu^\nu, \textq{which implies}
  \prod_{\nu=1}^n \binom{n}{\nu} = \prod_{\nu=1}^n \nu^\nu \bigg/ \prod_{\nu=1}^n \nu!.
\end{equation}
Second, by \eqref{eq:facFk} with $k=1$, we have
\begin{equation} \label{eq:facprod}
  \prod_{\nu=1}^n \nu!  \sims
  \MF_1 \, \MA \, (2\pi)^{\frac{1}{4}} \times
  \left( e^{-\frac{3}{2}} n \right)^{\binom{n+1}{2}}
  \left( 2\pi e^{-\frac{1}{2}} n \right)^{\frac{n}{2}} n^{\frac{5}{12}}
  \quad \text{as\ } n \to \infty.
\end{equation}
Third, we need the asymptotic formula~\eqref{eq:GK} of the hyperfactorial.
By~\eqref{eq:facpow}, we divide each side of~\eqref{eq:GK} by~\eqref{eq:facprod},
respectively, which easily yields the result.
\end{proof}

\begin{proof}[Proof of Theorem~\ref{thm:bin2}]
Let~$a, n \geq 1$. Since $\lim\limits_{a \to \infty} \MF_a = 1$, we deduce that
\[
  \lim_{a \to \infty} \MC_{a,1} = \lim_{a \to \infty}
  \frac{\MF_a}{\MF_{1} \, \MF_{a-1}} (2\pi)^{-\frac{1}{4}} = \MC_{\mathrm{row}}.
\]
Note that
\[
  a^{-\binom{n+1}{2}} \prod_{\nu=1}^n (a \nu)_{\nu}
  =\! \prod_{\nu=1}^n \left( a^{-\nu} (a \nu)_{\nu} \right)
  = 1 \cdot \left( 2 \cdot (2-\tfrac{1}{a}) \right)
  \cdot \left( 3 \cdot (3-\tfrac{1}{a}) \cdot (3-\tfrac{2}{a}) \right)
  \cdots \left( n \cdots (n-\tfrac{n-1}{a}) \right).
\]
Hence, we infer from \eqref{eq:facpow} that
\[
  \lim_{a \to \infty} a^{-\binom{n+1}{2}} \prod_{\nu=1}^n (a \nu)_{\nu}
  = \prod_{\nu=1}^n \nu^\nu = \prod_{\nu=1}^n (n)_{\nu},
  \textq{and thus}
  \lim_{a \to \infty} a^{-\binom{n+1}{2}} \prod_{\nu=1}^n \binom{a \nu}{\nu}
  = \prod_{\nu=1}^n \binom{n}{\nu}.
\]
Since the asymptotic constants and products coincide, respectively,
so do their asymptotic formulas. Obviously, the equality of the asymptotic
formulas can be shown directly. We leave the details to the reader.
\end{proof}


\section{Hirschhorn's theorem revisited}

Hirschhorn~\cite{Hirschhorn:2013} showed the following theorem,
apparently unaware of \cite{Kellner:2009} published a few years before.

\begin{theorem}[Hirschhorn~{\cite[Theorem]{Hirschhorn:2013}}]
Asymptotically, we have for $n \geq 1$ that
\[
  \prod_{k=0}^n \binom{n}{k} \sims \HC^{-1}
  \frac{e^{n(n+2)/2}}{n^{(3n+2)/6}(2\pi)^{(2n+1)/4}}
  \exp \left( - \sum_{\nu \geq 1} \frac{B_{\nu+1}+B_{\nu+2}}{\nu(\nu+1)}
  \frac{1}{n^\nu} \right) \quad \text{as\ } n \to \infty,
\]
where
\[
  \HC = \lim_{n \to \infty} n^{-\frac{1}{12}} \prod_{k=1}^{n}
  \left( k! \Big/ \sqrt{2\pi k} \left( \frac{k}{e} \right)^{\!\!k} \right)
  \approx 1.046335066770503180980950656977760\ndots.
\]
\end{theorem}

He gave a very lengthy and technical proof of the above theorem, but deriving
an asymptotic expansion by using divergent sums that involve Bernoulli numbers.
The constant $\HC$ was published by Hirschhorn as sequence \seqnum{A213080}
on the OEIS~\cite{OEIS}. By Theorem~\ref{thm:bin} and Table~\ref{tbl:Fk},
it finally turns out that
\[
  \HC = \MF_1 = (2\pi)^{\frac14} \, e^{\frac{1}{12}} / \MA^2.
\]
Note that $\HC^{-1}$ is not the complete asymptotic constant, compared with
Theorem~\ref{thm:bin}, since one of the above terms has a constant factor, namely,
\[
  (2\pi)^{(2n+1)/4} = (2\pi)^{\frac{1}{4}} \times (2\pi)^{\frac{n}{2}}.
\]


\section*{Acknowledgment}

We would like to thank the referee for useful suggestions that improved
the quality of the paper.


\bibliographystyle{amsplain}

\end{document}